\documentclass[9pt]{amsart}

\usepackage{amsmath,mathrsfs}
\usepackage{amssymb,amsfonts}
\usepackage{url,relsize,xcolor}
\usepackage{tikz,tikz-cd}
\usepackage{enumitem}
\usepackage{xr}
\usepackage[hidelinks]{hyperref}

\theoremstyle{plain}
\newtheorem{theorem*}{Theorem}
\newtheorem{theorem}{Theorem}
\newtheorem{prp}[theorem]{Proposition}

\newtheorem{corollary}[theorem]{Corollary}
\newtheorem{claim}[theorem]{Claim}
\newtheorem{subclaim}[theorem]{Subclaim}

\numberwithin{equation}{section}

\theoremstyle{definition}
\newtheorem{definition}[theorem]{Definition}

\theoremstyle{remark}

\DeclareMathOperator{\CH}{\mathsf{CH}}

\tikzset{
 symbol/.style={
 draw=none,
 every to/.append style={
 edge node={node [sloped, allow upside down, auto=false]{$#1$}}}
 }
}

\numberwithin{theorem}{section}

\newcommand{\bbN}{{\mathbb N}}

\newcounter{my_enumerate_counter}
\newcommand{\pushcounter}{\setcounter{my_enumerate_counter}{\value{enumi}}}
\newcommand{\popcounter}{\setcounter{enumi}{\value{my_enumerate_counter}}}

\newcommand{\norm}[1]{\left\lVert #1 \right\rVert}

\DeclareMathOperator{\OCA}{{\mathsf {OCA}}}

\DeclareMathOperator{\WEP}{{\mathsf {wEP}}}

\newcommand{\supp}{\mathrm{supp}}

\DeclareMathOperator{\MA}{{\mathsf {MA}_{\aleph_1}}}

\DeclareMathOperator{\ZFC}{{\mathsf {ZFC}}}

\newcommand{\cstar}{$\mathrm{C^*}$}

\title{The weak Extension Principle}

\author{Alessandro Vignati}
\address[AV]{
 Institut de Math\'ematiques de Jussieu - Paris Rive Gauche (IMJ-PRG)\\
 Universit\'e Paris Cit\'e\\ Institut Universitaire de France\\
 B\^atiment Sophie Germain\\
 8 Place Aur\'elie Nemours \\ 75013 Paris, France}
\email{vignati@imj-prg.fr}
\urladdr{http://www.automorph.net/avignati}

\author{Deniz Yilmaz}
\address[DY]{
Institut de Recherche en Informatique Fondamentale (IRIF)\\
 Universit\'e Paris Cit\'e\\
 B\^atiment Sophie Germain\\
 8 Place Aur\'elie Nemours \\ 75013 Paris, France}
\email{deniz.yilmaz@irif.fr}
\urladdr{https://denizyilmaz.fr}
\date{\today}

\begin{document}
\begin{abstract}
    We prove a rigidity result for maps between \v{C}ech--Stone remainders under fairly mild forcing axioms.
\end{abstract}
\maketitle
\section{Introduction}
This paper focuses on \v{C}ech--Stone remainders of locally compact topological spaces, spaces of the form $X^*=\beta X\setminus X$ where $\beta X$ is the \v{C}ech--Stone compactification of $X$, and maps between them.

A natural way to construct a continuous map between \v{C}ech--Stone remainders is to consider the restriction  of the compactification of a continuous map between the underlying spaces (perhaps forgetting about compact sets). To be precise, if $X$ and $Y$ are locally compact noncompact spaces, and $V_X\subseteq X$ is an open set with compact closure\footnote{The need of $V_X$ is justified, for example, by the fact that the spaces $\mathbb R^*$ and $((-\infty,0]\cup [1,\infty))^*$ are homeomorphic, but no continuous function $\mathbb R\to (-\infty,0]\cup[1,\infty)$ can induce a homeomorphism between them.}, a proper continuous map $X\setminus V_X\to Y$ extends to a map $\beta (X\setminus V_X)\to\beta Y$ whose restriction to $X^*$ gives a continuous map $X^*\to Y^*$. Does this construction recover all continuous maps between \v{C}ech--Stone remainders, at least without appealing to unnatural back-and-forth constructions? This question was first analyzed for the simplest space among all, $\mathbb N$ with the discrete topology, and it was given a negative answer: in \cite{dow2014non}, Dow constructed (in $\ZFC$) a nontrivial copy of $\bbN^*$ inside $\bbN^*$ which is nowhere dense, and therefore cannot come from a map $\bbN\to\bbN$ as above (further considerations on nontrivial copies of $\bbN^*$ were recently made in \cite{dow2024autohomeomorphismspreimagesmathbbn} and \cite{dow2024nontrivialcopiesn}). Conjecturally, this is all there is in $\ZFC$, and under some reasonably weak forcing axioms, all maps between \v{C}ech--Stone remainders arise as a blend of two maps as above. This intuition was formalised by Farah, who in \cite[\S4]{Fa:AQ} introduced the \emph{weak Extension Principle}. This principle was stated originally for pairs of zero-dimensional spaces; the following adapts it to the class of locally compact noncompact second countable spaces:
\begin{definition}
Let $X$ and $Y$ be locally compact noncompact second countable topological spaces. We say that $X$ and $Y$ satisfy the weak Extension Principle, and write $\WEP(X,Y)$, if the following happens for all pairs of naturals $d ,\ell \geq 1$:

For every continuous map $F\colon (X^*)^d\to (Y^*)^\ell$ there exists a partition into clopen sets $(X^*)^d =U_0\cup U_1$, an open set with compact closure $V_X\subseteq X$ such that $F[U_0]$ is nowhere dense in $(Y^*)^\ell$, and a continuous function $G\colon (\beta(X\setminus V_X))^d\to(\beta Y)^\ell$ which restricts to $F$ on $U_1$.

By $\WEP$ we denote the statement ``$\WEP(X,Y)$ holds whenever $X$ and $Y$ are locally compact, noncompact second countable spaces".
\end{definition}

The weak Extension Principle cannot be a consequence of $\ZFC$, as witnessed by the existence of nontrivial autohomeomorphisms of $\mathbb N^*$ under the Continuum Hypothesis $\CH$ (\cite{Ru}). More in depth, by a classical result of Parovi\v{c}enko (\cite{Pa:Universal}, $\CH$ implies that if $X$ is a zero-dimensional locally compact noncompact second countable space, then $X^*$ is homeomorphic to $\bbN^*$. All of these nontrivial homeomorphisms witness the failure of $\WEP$. Further, again under $\CH$, all \v{C}ech--Stone remainders (of second countable spaces) are continuous images of $\bbN^*$, and all connected ones are continuous images of $[0,1)^*$. A simple counting argument 
shows most of these maps cannot arise from continuous maps from $\bbN$ (or $[0,1)$) to the locally compact space of interest. In general, $\CH$ and back-and-forth constructions make essentially impossible to classify maps between \v{C}ech--Stone remainders.

Is it then possible that $\WEP$ is consistent? Instances of $\WEP$ were analysed first for $X=Y=\bbN$, and more generally to study maps between finite powers of $\mathbb N^*$ in different models of set theory (see \cite{Just:Omega^n, Fa:Powers, Just:WAT,vD:Prime,DoHa:Images} and \cite{vanMill.Three}, or \cite{Fa:Fourth} for more related questions). Farah, while formally introducing the $\WEP$, showed under some fairly mild forcing axioms (which are, Todorcevi\`c's $\OCA$ and Martin's Axioms $\MA$), $\WEP(X,Y)$ holds for pairs of countable locally compact noncompact spaces. Further results were obtained in \cite{FaMcK:Homeomorphisms} and \cite{vignati2018rigidity}, focusing on just homeomorphisms instead of general continuous maps. For a lengthy and detailed treatment on how forcing axioms impact the structure of homomorphisms between quotient structures, see \cite{farah2022corona}.

The main result of this article is that the weak Extension Principle holds unconditionally on the spaces of interest, again under the same forcing axioms assumptions.

\begin{theorem}\label{thm:main}
Assume $\OCA$ and $\MA$. Then $\WEP$ holds.
\end{theorem}

As a consequence, we show that the thesis of \cite[Theorem 5.1]{Fa:Dimension} holds (under $\OCA$ and $\MA$) for all pairs of locally compact noncompact second countable spaces: if $\kappa$ and $\lambda$ are cardinals with $\kappa<\lambda$ then there is no continuous surjection $(X^*)^\kappa\to(Y^*)^\lambda$. This result does not hold under $\CH$, as under this assumption $\bbN^*$ continuously surjects onto every compact topological space of weight $\leq\mathfrak c$, and specifically onto every finite power of every \v{C}ech--Stone remainder of a locally compact second countable topological space, while $[0,1)^*$ continuously surjects onto every finite power of every such connected \v{C}ech--Stone remainder.

The proof of our main result proceeds as follows: first, we use a result of Farah (\cite{Fa:Dimension}) on maps between powers of \v{C}ech--Stone remainders (more generally, $\beta\bbN$-spaces) to simplify the situation and show that it is enough to show $\WEP$ for maps where $d=\ell=1$ (Proposition~\ref{prop:dimension}). Secondly, we rely on Gel'fand's duality, and study $^*$-homomorphisms between coronas of abelian \cstar-algebras. We build on a strong lifting theorem proved in \cite{vignati2018rigidity} (see Proposition~\ref{prop:oldrigidity}) which gives that, under forcing axioms, $^*$-homomorphisms between abelian corona \cstar-algebras have well-behaved liftings on certain blocked subalgebras. The main content of our proof is to glue and analyse such liftings. This approach is fairly different from the one of \cite[\S4]{Fa:AQ}, where Boolean algebras of clopen sets were used. Once the dimension increases, Stone duality is replaced by Gel'fand's one, as clopen sets do not recover completely the topology in this case.

For all background material on \cstar-algebras theory we refer the reader to \cite{Black:Operator}, or to \cite{Fa:STCstar} for the connections between \cstar-algebras and set theory. For notions in topology, we refer to the standard book of Munkres (\cite{Munkres.Topology}) and to \cite{GillManJerison.Rings} more specifically for \v{C}ech--Stone compactifications and remainders.

\subsection*{Acknowledgements}
The authors would like to thank I. Farah for useful conversations, and the anonymous referee for helpful comments.
\subsection*{Funding}
This research originated from DY's Master thesis, written under the guidance of AV at Universit\'e Paris Cit\'e. In this period, DY was supported by an INSMI (Institut National des Sciences Mathématiques et de leurs Interactions) grant for master students.

\section{Proof of the main result}

This main section is fully dedicated to the proof of  Theorem \ref{thm:main}. As mentioned, our first step is to allow us to only consider maps between \v{C}ech--Stone remainders. This is due to a powerful result of Farah. Some context: If $d\geq 1$, a function between topological spaces $F \colon X^d \to Y$ is said to depend on at most one coordinate on $U \subseteq X^d$ if there exists $j \leq d$ and $H \colon X \to Y$ such that $F \restriction U = H \restriction \pi_j (U)$, $\pi_j\colon X^d\to X$ being the projection on the $j$-th coordinate.  A rectangle in $X^d$ is a set of the form $\prod_{i\leq d}A_i$, where $A_i\subseteq X$. The following is a reformulation of \cite[Theorem 3]{Fa:Dimension}\footnote{Theorem 3 in \cite{Fa:Dimension} works even if $d$ is an infinite cardinal. Further, the codomain space $Y$ needs to be a $\beta\bbN$-space, and not necessarily a remainder space.}.
\begin{theorem}\label{theorem:products}
Let $K$ be a compact space, $Y$ be locally compact, noncompact and second countable, $d\geq 1$, and let $F\colon K^d\to Y^*$ be a continuous map. Then $K^d$ can be covered by finitely many clopen rectangles on each of which $F$ depends on at most one coordinate.
\end{theorem}

\begin{prp} \label{prop:dimension}
Suppose the $\WEP$ holds for all continuous maps between \v{C}ech--Stone remainders. Then the $\WEP$ holds.
\end{prp}

\begin{proof}
The statement of the proposition asserts that if we can prove the weakening of the $\WEP$ where we only consider maps between \v{C}ech--Stone remainders, i.e. $d=\ell=1$, then the $\WEP$ holds.

At first, we show that we can assume $\ell=1$. For this, assume the $\WEP$ holds for all maps $(X^*)^d\to Y^*$ (where $X$ and $Y$ satisfy the appropriate hypotheses, and $d\geq 1$ is arbitrary). Let $F\colon (X^*)^d\to (Y^*)^\ell$ be a continuous map, and, for $i\leq \ell$, compose $F$ with the projection on the $i$-th coordinate maps. This gives continuous maps $F_i=\pi_i\circ F\colon(X^*)^d\to Y^*$. Applying the weak Extension principle to these maps, we get, for $i\leq \ell$, a clopen decomposition $(X^*)^d = U_0^i \cup U_1^i$ and an open set with compact closure $V_X^i$ such that $F_i[U_0^i]$ is nowhere dense and there is a continuous $G_i\colon (\beta X\setminus V_X^i)^d\to\beta Y$ which restricts to $F_i$ on $U_1^i$. The sets $U_0=\bigcup_i U_0^i$, $U_1=\bigcap_i U_1^i$, $V_X=\bigcup_i V_X^i$ and the map $G\colon (\beta X\setminus V_X^i)^d\to(\beta Y)^\ell$ sending $x$ to $(G_1(x),\ldots,G_\ell(x))$ witness that $F$ satisfies the $\WEP$.

Let us now suppose we have a continuous map of the form $F \colon (X^*)^d \to Y^*$, and apply Theorem~\ref{theorem:products}. Then there exists a finite cover of $(X^*)^d$ by clopen rectangles $R_1,\dots ,R_n$. Let $R_i = A_1^i \times \cdots \times A_d^i$. Since each $R_i$ is clopen in the compact space $(X^*)^d$, we can assume each $A^i_j\subseteq X^*$ is clopen, hence we can find open sets $Z_{i,j}\subseteq X$ such that $A_j^i=Z_{i,j}^*$. For every $i$, let $j(i)$ be such that $F\restriction R_i$ depends on the $j(i)$-th coordinate. Hence we can find a continuous $H_i\colon A_{j(i)}^i=Z_{i,j(i)}^*\to Y^*$. Applying the $\WEP$ to each $H_i$, for $i\leq n$, we get open sets with compact closure $V_i\subseteq Z_{i,j(i)}$, partitions into clopen sets $Z_{i,j(i)}^*=U_0^i\cup U_1^i$ and continuous maps $G_i\colon \beta (Z_{i,j(i)}\setminus V_i)\to \beta Y$ lifting $H_i$ on $U_1^i$. For $i\leq n$, define
\[
\tilde G_i\colon \beta Z_{i,1}\times\cdots\times \beta(Z_{i,j(i)}\setminus V_i)\times\cdots\times \beta Z_{i,d}\to \beta Y
\]
by mapping a tuple to the $G_i$ image of its $j(i)$-th entry. Fix $V_X=\bigcup V_i$. Gluing all of these maps together gives a continuous lift on 
\[
U_1=\bigcup_{i\leq n}(Z_{i,1}^*\times\cdots\times \underbrace{U_1^i}_{j(i)}\times\cdots\times Z_{i,d}^*)
\]
while the $F$-image of the complement of $U_1$ is a finite union of the nowhere dense sets $F[Z_{i,1}^*\times\cdots\times U_0^i\times\cdots\times Z_{i,d}^*]$. This concludes the proof.
 \end{proof}

From now on, we only deal with continuous maps between \v{C}ech--Stone remainders (i.e., we assume that $d=\ell=1$ in the definition of the $\WEP$). We approach these via Gel'fand's duality, by looking at the \cstar-algebras $C(X^*)$ and $C(Y^*)$ and at $^*$-homomorphisms between them. In this article, we only consider abelian \cstar-algebras, i.e., algebras of continuous functions on locally compact topological spaces where the topology is induced by the $\sup$-norm and operations are defined pointwise. For an introduction and basic terminology, see \cite[II.2]{Black:Operator}.

Algebras of the form $C(X^*)$ are typical examples of corona \cstar-algebras (\cite[II.7]{Black:Operator} or \cite[\S 13]{Fa:STCstar}). The structure of $^*$-homomorphisms between corona algebras and their behaviour under forcing axioms was extensively studied in \cite{mckenney2018forcing} and \cite{vignati2018rigidity}, after the use of forcing axioms to analyse massive quotients of \cstar-algebras was popularised by the seminal \cite{Fa:All}.

If $X$ and $Y$ are locally compact noncompact second countable topological spaces then Gelf'and's duality gives that $C(\beta X)\cong C_b(X)$, the latter being the algebra of bounded continuous functions on $X$, and that $C(X^*)\cong C_b(X)/C_0(X)$, where $C_0(X)$ is the algebra of continuous functions on $X$ vanishing at infinity. We denote by $\pi_X$  the canonical quotient map $\pi_X\colon C_b(X)\to C(X^*)$.

Since our spaces are second countable and locally compact, they are $\sigma$-compact (alternatively, the \cstar-algebras $C_0(X)$ and $C_0(Y)$ are $\sigma$-unital). The following notation is set up to match the hypotheses of \cite[Theorem 4.3]{vignati2018rigidity}.
We let 
\begin{itemize}
\item $Y_n$ be an increasing sequence of open subsets of $Y$ such that $Y=\bigcup Y_n$, each $Y_n$ has compact closure, and $\overline Y_n\subseteq Y_{n+1}$ for all $n\in\bbN$. We assume each $Y_{n+1}\setminus \overline Y_n$ is not empty, and that $Y_0=\emptyset$.
\item $U_{n}^e=Y_{10n+7}\setminus \overline Y_{10n}$ and $U_{n}^o=Y_{10n+12}\setminus \overline Y_{10n+5}$.
\end{itemize}
The letters $e$ and $o$ are for even and odd partition. Recall (\cite[II.3.4]{Black:Operator}) that if $A$ is a \cstar-algebra, a \cstar-subalgebra $B\subseteq A$ is hereditary if for all positive $a\leq b$, $b\in B$ implies that $a\in B$, where $a\in A$ is positive if $0\leq a$ (in the abelian setting, $a$ takes only values in $[0,\infty)$).

We have the following properties:
\begin{enumerate}
\item for every $n\neq m$ and $i\in \{e,o\}$,  $\overline U_n^i\cap \overline U_m^i=\emptyset$;
\item for each $i\in\{e,o\}$, $\prod_nC_0(U_n^i)$ is a hereditary \cstar-subalgebra of $C(\beta Y)$ such that $(\prod_nC_0(U_n^i))\cap C_0(Y)=\bigoplus_nC_0(U_n^i)$. 
\item $C_b(Y)=\prod_nC_0(U_n^e)+\prod_nC_0(U_n^o)$, and consequently 
\[
C(Y^*)=\prod_nC_0(U_n^e)/\bigoplus_nC_0(U_n^e)+\prod_nC_0(U_n^o)/\bigoplus_nC_0(U_n^o).
\]
\end{enumerate}
Note that the sum above is not a direct sum, as $U_n^e\cap U_n^o\neq\emptyset$ for all $n$. 

If $i\in\{e,o\}$ and $f\in \prod_nC_0(U_n^i)$, since the sets $U_n^i$, for $n\in\bbN$, are pairwise disjoint, we can view $f$ as a sequence $(f_n)$, where $f_n\in C_0(U_n^i)$. If $S\subseteq\bbN$, we write $f_S$ for the element of $\prod_nC_0(U_n^i)$ defined by 
\[
(f_S)_n=\begin{cases}
    f_n&\text{ if }n\in S\\
    0 & \text{ else}.
\end{cases}
\]
For fixed $S\subseteq\bbN$ and $i\in\{e,o\}$, the map $f\mapsto f_S$  gives a $^*$-homomorphism $\prod_nC_0(U_n^i)\to \prod_{n\in S}C_0(U_n^i)$, dual to the inclusion $\bigcup_{n\in S}U_n^i\subseteq\bigcup_n U_n^i$.

Recall, a $^*$-homomorphism between nonunital \cstar-algebras $\psi\colon A\to B$ is nondegenerate (\cite[II.7.3.8]{Black:Operator}) if the ideal generated by the image of $\psi$ is dense (in norm topology). This is equivalent to ask that approximate unit for $A$ is sent to an approximate unit of $B$. In the abelian setting, Gel'fand's duality asserts that nondegenerate $^*$-homomorphisms are dual to proper continuous maps. (Recall, a map between topological spaces is proper if inverse images of compact subsets are compact.)

The following is derived from \cite{vignati2018rigidity}. For the reader's convenience, we shortly sketch its proof, referring heavily to the notation of \cite{vignati2018rigidity}.

\begin{prp}\label{prop:oldrigidity}
Assume $\OCA$ and $\MA$. Suppose that $\Phi\colon C(Y^*)\to C(X^*)$ is a $^*$-homomorphism. Then there are a natural $\bar n$, open subsets of $X$ with compact closure $V_n^e$ and $V_n^o$, for $n\geq \bar n$, and nondegenerate $^*$-homomorphisms 
\[
\psi_n^i\colon C_0(U_n^i)\to C_0(V_n^i)\text{ for }i\in \{e,o\}
\]
such that:
\begin{enumerate}
\item\label{proplift1} For $i\in \{e,o\}$ and $n\neq m$, $\overline V_n^i\cap \overline V_m^i=\emptyset$, and, if $n>\bar n$,
\[
\overline {V_{n}^e}\subseteq V_{n-1}^o\cup V_{n}^e\cup V_n^o, \,\, \overline {V_{n}^o}\subseteq V_{n}^e\cup V_{n}^o\cup V_{n+1}^e.
\]
\item\label{proplift2} Every compact subset of $X$ intersects only finitely many of the sets $\{V_n^i\mid n\geq \bar n, i\in\{e,o\}\}$.
\item For $i\in \{e,o\}$, 
\[
\Psi^i=\sum_{n\geq \bar n}\psi_n^i\colon \prod_{n\geq \bar n} C_0(U_n^i)\to\prod_{n\geq \bar n} C_0(V_n^i)
\]
is a $^*$-homomorphism such that for every $f=(f_n)\in \prod_{n\geq \bar n}C_0(U_n^i)$ there is a nonmeager ideal $\mathcal I_f\subseteq\mathcal P(\bbN)$ such that
\[
\pi_X(\Psi^i(f_S))=\Phi(\pi_Y(f_S))
\]
for all $S\in\mathcal I_f$, and
\item\label{prp:oldc3} If $\gamma_n^i\colon V_n^i\to U_n^i$ is the proper continuous map dual to $\psi_n^i$, then $\gamma_n^e$ and $\gamma_m^o$ (equivalently, the $^*$-homomorphisms $\psi_n^e$ and $\psi_m^o$) agree on their common domain.
\end{enumerate}
\end{prp}
\begin{proof}
Write $X$ as a increasing union of nonempty open sets $X=\bigcup X_n$ where for each $n$ we have that $\overline{X_n}\subseteq X_{n+1}$, $X_{n+1}\setminus \overline{X_n}$ is nonempty, and each $\overline{X_n}$ is compact. We furthermore ask that if $K\subseteq X$ is compact then $K\subseteq X_n$ for some $n$. This can be done since $X$ is $\sigma$-compact. Writing $B$ for $C_0(X)$, we let $e_n^B\in C_0(X)$ be a positive contraction\footnote{A contraction in a \cstar-algebra is an element whose norm is $\leq 1$.} such that $e_n^B[\overline {X_n}]=1$ and $e_n^B[X\setminus X_{n+1}]=0$, so that $e_{n+1}^Be_n^B=e_n^B$ for all $n$. 

We are ready to construct the sets $V_n^i$. For convenience, we only give the details on how to construct the sets $V_n^e$; the `odd' sets can be constructed in the same way after appropriately re-indexing. 

Let $e_n^A$ be an approximate identity of positive contractions for $A=C_0(Y)$ such that $e_n^A[\overline Y_n]=1$ and $e_n^A(y)\neq 0$ if and only if $y\in Y_{n+1}$, and let $J_n=[10n-1, 10n+7]$, so that the algebra $A_{J_n}$ given in \cite[Notation 3.3]{vignati2018rigidity} is precisely $C_0(U_n^e)$. Applying the results of \S3 and \S4 of \cite{vignati2018rigidity}, and specifically Lemma 4.2 therein, we get mutually orthogonal\footnote{Meaning that if $n\neq m$ then for all $f\in C_0(U_n^e)$ and $f'\in C_0(U_m^e)$ we have $\alpha_n(f)\alpha_m(f')=0$.} functions $\alpha_n\colon C_0(U_n^e)\to C_0(X)$ such that the function 
\[
\Gamma^e=\sum\alpha_n\colon \prod_n C_0(U_n^e)\to C_b(X)
\]
has the property that for every $f=(f_n)\in\prod_n C_0(U_n^e)$ we have a nonmeager ideal $\mathcal I_f$ such that
\[
\pi_X(\Gamma^i(f_S))=\Phi(\pi_Y(f_S))
\]
for all $S\in\mathcal I_f$. The mutually orthogonal $\alpha_n$s come with two sequences of naturals $j_n<k_n$ such that the range of $\alpha_n$ is contained in $(e_{k_n}^B-e_{j_n}^B)C_0(X)(e_{k_n}^B-e_{j_n}^B)$ and $\lim_n j_n= \infty$. This is the key property in showing condition~\eqref{proplift2}, as it gives that if $K\subseteq X$ is compact then the set
\[
\{n\mid\exists f\in C_0(U_n^e) \exists x\in K \, (\alpha_n(f)(x)\neq 0)\}
\]
is finite. 

We will now slightly modify the maps $\alpha_n$, which themselves are not $^*$-homomorphisms. By Ulam stability for approximate maps between abelian \cstar-algebras (a result proved by S\v{e}rml in \cite{Semrl.USAbel}, see also \cite[\S5]{vignati2018rigidity} or, for details, \cite[\S5.4 and Theorem 5.18]{farah2022corona}) shows that we can find $\bar n$ and mutually orthogonal $^*$-homomorphisms $\psi_n^e\colon C_0(U_n^e)\to (e_{k_n}^B-e_{j_n}^B)C_0(X)(e_{k_n}^B-e_{j_n}^B)$, for $n\geq\bar n$ such that $\Psi^i=\sum\psi_n^e$ still lifts elements of $\prod_nC_0(U_n^e)$ on nonmeager ideals. The sets $V_n^e$ are now constructed from the maps $\psi_n^e$ as in \cite[\S4.1]{vignati2018rigidity}\footnote{Injectivity of the maps $\psi_n^e$ is stated as an assumption there, but it is not used there.}, by setting
\[
V_n^e=\bigcup_{f\in C_o(U_n^e)}\supp(\psi^e_n(f)).
\]
Since $\lim_n j_n=\infty$ and $V_n^e\subseteq X_{k_n}\setminus X_{j_n}$, if a $K$ is compact then $K$ can intersect at most finitely many of the $V_n^e$s. This shows condition ~\eqref{proplift2}. The other properties are verified exactly as in \S4.1 in \cite{vignati2018rigidity}. This concludes the construction of the sets $V_n^e$, and the proof of the proposition.
\end{proof}

If $f$ and $g$ are elements of $C_b(X)$, we write $f=_{C_0(X)}g$ for $f-g\in C_0(X)$.

\begin{proof}[Proof of Theorem~\ref{thm:main}]
  
By Proposition~\ref{prop:dimension} it is sufficient to prove $\WEP$ for maps between \v{C}ech--Stone remainders (i.e., when $d=\ell=1$). Fix then a continuous function $F\colon X^*\to Y^*$, and let $\Phi\colon C(Y^*)\to C(X^*)$ be the dual $^*$-homomorphism. Let $\tilde\Phi\colon C_b(Y)\to C_b(X)$ be any set theoretic lift of $\Phi$. All of our notation will be as in  Proposition~\ref{prop:oldrigidity} and the discussion preceding it.

Let 
\[
V_1=\bigcup_{n\geq \bar n}V_n^e\cup V_n^o\text{ and } W=\bigcup_{n\geq \bar n}U_n^e\cup U_n^o.
\]
\begin{claim} 
$V_1$ is, modulo compact, clopen.
\end{claim}
\begin{proof}
We show that $\overline V_1\setminus V_1$ is contained in the compact set $\overline V_{\bar n}^e$. Combining conditions~\eqref{proplift1} and \eqref{proplift2} of Proposition~\ref{prop:oldrigidity} gives that 
\[
\overline {V_1}=\overline{\bigcup_{n\geq\bar n} V_n^e\cup V_n^o} = \bigcup_{n\geq\bar n} \overline{V_n^e}\cup\bigcup_{n\geq\bar n}\overline{V_n^o}\subseteq \overline{V_{\bar n}^e}\cup \bigcup_{n\geq\bar n} V_n^e\cup V_n^o,
\]
where the equality $\overline{\bigcup_{n\geq\bar n} V_n^e\cup V_n^o} = \bigcup_{n\geq\bar n} \overline{V_n^e}\cup\bigcup_{n\geq\bar n}\overline{V_n^o}$ is given by the fact that compact subsets of $X$ can only intersect finitely many of the sets $V_n^e\cup V_n^o$. This shows that $\overline V_1\setminus V_1$ is contained in the compact set $\overline{V_{\bar n}^e}$.
\end{proof}
(As a side note, a deep understanding of \cite[\S4.1]{vignati2018rigidity} shows that if the spaces of interest can be written as a union of pairwise disjoint compact open sets, $V_1$ can be chosen to be clopen. We shall not need this fact.)

As the continuous proper maps  $\gamma_n^i$ agree on their common domains, the function
\[
\Gamma=\bigcup_{n\geq \bar n}\gamma_n^e\cup\gamma_n^o\colon V_1\to W
\]
is a well-defined continuous function. Furthermore, $\Gamma$ is proper: if $K\subseteq W$ is compact, then there exists $\bar n'$ such that $K\subseteq\bigcup_{\bar n\leq n\leq \bar n'}U_n^e\cup U_n^o$, hence $\Gamma^{-1}[K]$ is a closed subset of $\bigcup_{\bar n\leq n\leq \bar n'}V_n^e\cup V_n^o$. As the latter has compact closure, $\Gamma^{-1}[K]$ is compact. Since $\Gamma$ is proper, it induces a $^*$-homomorphism 
\[
\Psi\colon C_b(W)\to C_b(V_1)
\]
mapping $C_0(W)$ to $C_0(V_1)$ and whose dual is $\beta\Gamma\colon \beta V_1\to \beta W\subseteq\beta Y$. 

We are ready to define the necessary sets and maps to prove the $\WEP$: first, let $V_X\subseteq X$ be an open set with compact closure such that $V_1$ is clopen in $X\setminus V_X$, and set $V_0=(X\setminus V_X)\setminus V_1$. Note that  $\beta (X\setminus V_X)=\beta V_1\sqcup\beta V_0$ and therefore $U_0:=V_0^*$ and $U_1:=V_1^*$ is a clopen partition of $X^*$, by compactness of $\overline{V_X}$. Pick a point in $y\in\beta Y$, let $G'\colon \beta V_0\to \{y\}$ be the constant map, and let $G=\beta\Gamma\cup G'$.

We are now going to show these objects satisfy $\WEP(X,Y)$, that is, that $G\restriction U_1=F$ and that $F[U_0]$ is nowhere dense.
\begin{claim}\label{claim1}
$G\restriction U_1=F$.
\end{claim}
\begin{proof}
First, write $\Phi$ as $\Phi=\Phi_0\oplus\Phi_1$, where 
\[
\Phi_i(f)=\chi_{U_i}\Phi(f), \,\, i\in \{0,1\},
\]
$\chi_{U_i}\in C(X^*)$ being the characteristic function of $U_i$. We have the following diagram:
\begin{center}
\begin{tikzcd}
 & C(U_0)  \arrow[dr] & \\
C(Y^*) \arrow[dr, "\Phi_1" ] \arrow[ur, "\Phi_0"] \arrow[rr, "\Phi = \Phi_0 \oplus \Phi_1"] & & C(X^*).  \\
 & C(U_1) \arrow[ur] &
\end{tikzcd}
\end{center}
Both $\Phi_0$ and $\Phi_1$ are $^*$-homomorphisms, as $C(X^*)$ is abelian. To show that $G\restriction U_1=F$, it is enough to show that the $^*$-homomorphism $\Psi$ lifts $\Phi_1$.

Since every function $f\in C_b(Y)$ can be written (modulo $C_0(Y)$) as a sum $f=f_e+f_o$ where, for $i\in \{e,o\}$, $f_i\in \prod_{n\geq \bar n}C_0(U_n^i)$, it is enough to show $\Psi$ lifts $\Phi_1$ on $\prod_{n\geq \bar n}C_0(U_n^i)$, again for $i\in \{e,o\}$. For convenience, let $i=e$, and fix mutually orthogonal positive contractions $g_n\in C_b(X)$ such that for all $n\geq \bar n$
\[
g_n[V_n^e]=1\text{ and }\supp (g_n) \subseteq V^o_{n-1} \cup V^e_n \cup V^o_n.
\]
Let $g_n=0$ if $n<\bar n$. Such functions exist as distinct $V_n^e$s have disjoint closures. For $S\subseteq\bbN$, let $g_S=\sum_{n\in S}g_n$. Note that, even if $V_1$ is not clopen, it is clopen modulo compact, and therefore the function $\chi_{V_1}$ is equal, modulo $C_0(X)$, to a continuous function.
\begin{subclaim}
For every $f\in \prod_{n\geq \bar n}C_0(U_n^e)$ and $S\subseteq\bbN$ we have that 
\[
g_S\tilde\Phi(f)=_{C_0(X)}\chi_{V_1}\tilde\Phi(f_S).
\]
\end{subclaim}
\begin{proof}
Fix $f\in\prod_{n\geq \bar n}C_0(U_n^e)$ and $S\subseteq\bbN$. We can suppose $S$ is infinite, otherwise both $g_S$ and $\tilde\Phi(f_S)$ belong to $C_0(X)$. We will show that for every $\varepsilon>0$ there is a compact set $K$ such that 
\begin{equation}\label{eqn1}
\text{if }(x\in V_1\setminus K\text{ and } |\tilde\Phi(f_S)(x)|>\varepsilon)\text{ then }x\in \bigcup_{n\in S}V_n^e.
\end{equation}
This gives that $g_S\tilde \Phi(f_S)=_{C_0(X)}\chi_{V_1}\tilde \Phi(f_S)$. As $S$ is arbitrary, the same reasoning applied to $\bbN\setminus S$, gives that $g_S\tilde\Phi(f_{\bbN\setminus S})\in C_0(X)$. Altogether, we get that
\[
\chi_{V_1}\tilde\Phi(f_S)=_{C_0(X)}g_S\tilde\Phi(f_S)=_{C_0(X)}g_S\tilde\Phi(f_S)+g_S\tilde\Phi(f_{\mathbb N\setminus S})=_{C_0(X)}g_S\tilde\Phi(f),
\]
which is the claim.

Suppose \eqref{eqn1} fails for $\varepsilon>0$. Then there is a sequence of points $x_k\in V_1\setminus(\bigcup_{n\in S}V_n^e)$ such that $|\tilde\Phi(f_S)(x_k)|>\varepsilon$ and $x_k\to\infty$ as $k\to\infty$ (meaning for every compact $K\subseteq X$ we have that $K\cap \{x_k\}$ is finite). Passing to a subsequence, we can assume that there is $i\in \{e,o\}$ such that $\{x_k\}\subseteq \bigcup_n V_n^i$. Let $n_k$ be such that $x_k\in V_{n_k}^i$. As $x_k\to\infty$ as $k\to\infty$, we can assume, by further shrinking the sequence $\{x_k\}$, that $\{n_k\}$ is strictly increasing and that $n_0>\bar n$. 

Since the homomorphisms $\psi_n^i$ are nondegenerate, then for all $n>\bar n$ and $i\in \{e,o\}$,
\begin{align*}
V_n^i&=\bigcup\{\supp(\psi_n^i(h))\mid h\in C_0(U_n^i)\}\\&=\{x\in X\mid\exists h\in C_0(U_n^i), 0\leq h\leq 1, \psi_n^i(h)(x)=1\}\\
&=\{x\in X\mid\exists h\in C_0(U_n^i), 0\leq h\leq 1, \Psi(h)(x)=1\}.
\end{align*}
Hence, for each $k$ we can find positive contractions $h_{n_k}\in C_0(U_{n_k}^i)$ such that 
\[
\psi_{n_k}^i(h_{n_k})(x_k)=\Psi(h_k)(x_k)=1.
\]
Since $x_k\notin\bigcup_{n\in S}V_n^e=\Gamma(\bigcup_{n\in S}U_n^e)$, we can further assume that $\supp(h_{n_k})\cap \bigcup_{n\in S}U_n^e=\emptyset$. 

Set $h_j=0$ if $j\notin \{n_k\mid k\in\mathbb N\}$, and let 
\[
h=\sum h_k\in \prod_{n\geq \bar n} C_0(U_n^i).
\]
Notice that for every $T\subseteq\bbN$ we have $\Psi(h_T)(x_k)=1$ whenever $k$ is such that $n_k\in T$ and $h_Tf_S=0$, since $\supp(f_S)\subseteq\bigcup_{n\in S}U_n^e$ which is disjoint from $\supp(h)\supseteq \supp(h_T)$. 

Let $\mathcal I_h$ be the nonmeager ideal given to us from Proposition~\ref{prop:oldrigidity}, and let $T$ be an infinite set in $\mathcal I_h$ such that $T\subseteq\{n_k\}$, so that $\Psi(h_T)=_{C_0(X)}\tilde\Phi(h_T)$. Since $|\tilde\Phi(f_S)(x_k)|>\varepsilon$ then 
\[
\limsup_{k\to\infty}|\tilde\Phi(f)\tilde\Phi(h_T)(x_k)|=\limsup_{k\to\infty}|\tilde\Phi(f)\Psi(h_T)(x_k)|\geq \varepsilon.
\]
Since $x_k\to\infty$ as $k\to\infty$, this contradicts that $\tilde\Phi(f)\tilde\Phi(h_T)\in C_0(X)$, as $\tilde\Phi$ lifts the $^*$-homomorphism $\Phi$ and $h_T$ and $f_S$ are orthogonal.
\end{proof}
By the subclaim, $g_\bbN\tilde\Phi(f)$ is a lift for $\Phi_1(f)$.
Let 
\[
\mathcal J_f=\{S\subseteq\bbN\mid g_S(\Psi(f)-\tilde\Phi(f))\in C_0(X)\}.
\]
Since $g_S\Psi(f)=\Psi(f_S)$ for all $S\subseteq\bbN$, and $\Psi(f_S)=_{C_0(X)}\tilde\Phi(f_S)$ whenever $S\in\mathcal I_f$, then $\mathcal I_f\subseteq\mathcal J_f$, and the latter is therefore nonmeager and contains all finite sets. Since the association $S\mapsto g_S$ is continuous (when $C_b(X)$ is given the strict topology), $C_0(X)$ is Borel in $C_b(X)$, and $\Psi(f)$ and $\tilde\Phi(f)$ are fixed, $\mathcal J_f$ is a Borel ideal. Since the only Borel nonmeager ideal containing $\mathrm{Fin}$ is $\mathcal P(\bbN)$, $\mathcal J_f=\mathcal P(\bbN)$, hence $g_\bbN(\Psi(f)-\tilde\Phi(f))\in C_0(X)$. Since $g_\bbN\Psi(f)=\Psi(f)$ and $g_\bbN\tilde\Phi(f)$ lifts $\Phi_1(f)$, we have that $\Psi(f)$ lifts $\Phi_1(f)$, which is our thesis. The claim is then proved.
 \end{proof}
\begin{claim}\label{claim2}
$F\restriction U_0$ has nowhere dense range.     
\end{claim}
\begin{proof}
Since $F$ is a continuous map between compact spaces, it is closed, hence so is $F[U_0]$. Suppose for a contradiction that there is a nonempty open $U \subseteq F[U_0]\subseteq Y^*$. Let $O\subseteq \beta Y$ be open such that $O\cap Y^*=U$, and let $O'=O\cap Y$. Since the sets $\{U_n^i\mid i\in \{e,o\}, n\in\bbN\}$ cover $Y$, we can find an infinite $S\subseteq\bbN$ and $i\in \{e,o\}$ such that $O'\cap U_n^i\neq\emptyset$ for all $n\in S$. For $n\in S$, let $h_n$ be a positive contraction whose support is in $U_n^e$ and such that there is $x\in O'$ with $h_n(x)=1$. Let $h=\sum h_n$. Note that for every infinite $T\subseteq S$ we have that $\norm{\Phi(h_T)\chi_{U_0}}=1$. Find $T\subseteq S$ infinite such that $\Psi$ lifts $\Phi$ on $h_T$, and notice that $\chi_{U_0}\pi_X(\Psi(h_T))=0$. This is a contradiction.
\end{proof}
Claims~\ref{claim1} and \ref{claim2} together show that $G$, $U_0$ and $U_1$ satisfy the hypotheses of $\WEP(X,Y)$. This concludes the proof of the theorem.
\end{proof}

The following is a consequence of a result of Farah: 
\begin{corollary}
Assume $\OCA$ and $\MA$. Let $\kappa<\lambda$ be cardinals, and suppose that $X$ and $Y$ are locally compact noncompact second countable spaces. Then there is no continuous surjection $(X^*)^\kappa \to (Y^*)^\lambda$.
\end{corollary}
\begin{proof}
This is a consequence of Theorem 5.1 in \cite{Fa:Dimension}, which asserts that the thesis holds if we have $\WEP$. The result then follows from Theorem~\ref{thm:main}.
\end{proof}
The $\WEP$ has consequences on the existence of surjections from \v{C}ech--Stone remainders to powers of them: while $\CH$ implies that $\mathbb N^*$ surjects onto every compact space of density $\aleph_1$, and therefore onto every power of remainder of every (reasonable small) space, the same cannot happen if $\WEP$ holds, where for example that $(\mathbb N^*)^d$ surjects onto $(\mathbb N^*)^\ell$ if and only if $d\geq \ell$ (see \cite[\S4.6]{Fa:AQ}). In future work we intend to explore these phenomena, along with extension principles in the noncommutative setting.



\bibliographystyle{amsplain}
\bibliography{bibliography}

\end{document}